\documentclass[11pt]{article}
\usepackage[utf8]{inputenc}
\usepackage[english]{babel}
\usepackage{amsmath, amssymb,amsthm, bm}
\usepackage{mathdots}

\newtheorem{theorem}{Theorem}
\newtheorem{lemma}[theorem]{Lemma}

\newtheorem{corollary}[theorem]{Corollary}
\theoremstyle{nonumberplain}

\title{Restricted Cartan Type Lie Algebras and the Coadjoint Representation}
\author{Martin Mygind}
\date{\today}

\begin{document}
\maketitle

\begin{abstract}
Let $L$ be a restricted Cartan type Lie algebra over an algebraically closed field $k$ of characteristic $p>3$, and let $G$ denote the automorphism group of $L$. We prove that there are no nontrivial invariants of $L^*$ under the coadjoint action, i.e., $k[L^*]^G=k$. This property characterises the Cartan type algebras among the restricted simple Lie algebras. 
\end{abstract}

\section{Introduction} 
In his 1992 paper \cite{P} Premet showed, among many other things, that the invariant theory of the Witt-Jacobson algebras $W(n)$ is largely analogous to the corresponding theory for simple Lie algebras in characteristic zero. Namely, the algebra of invariants $k[W(n)]^G$ under the action of the automorphism group $G$ of $W(n)$ is generated by $n$ algebraically independent polynomials, and is thus isomorphic to a polynomial ring in $n$ variables. Chevalley's Restriction Theorem asserts that the same is true for any simple Lie algebra in characteristic zero, if we replace the automorphism group by the adjoint group. Recently, Premet's results were extended to other restricted Lie algebras of Cartan type in \cite{BFS} by introducing a notion of Weyl group for these algebras.

Now, for any Lie algebra $L$, the automorphism group of $L$ also acts on the dual space $L^*$ in a canonical way. For semisimple $L$ in characteristic zero the Killing form induces a $G$-isomorphism between $L$ and $L^*$, so here the problem of finding the invariants of $L^*$ under the $G$-action is reduced to the known calculation of $k[L]^G$. Not so for a restricted Lie algebra $L$ of Cartan type, where the Killing form is identically zero. Here we prove that there are \emph{no} nontrivial invariants of $L^*$, which of course represents a radical departure from the characteristic zero theory (this was proved for the simplest case $L=W(1)$ in \cite{M}).  

\section{Restricted Cartan Type Lie Algebras}
We will briefly review the basic definitions and results regarding restricted Cartan type Lie algebras, following 2.8-2.11 in \cite{BGP} (which again refers to \cite{SF} for many proofs). Note, however, that \cite{BGP} takes a more general approach by using divided power algebras, whereas our approach is more concrete. The difference amounts to nothing more than a scaling of basis elements, so though some formulas might change, everything is essentially the same. Let $k$ be an algebraically closed field of characteristic $p>3$, and let $A(n)=k[X_1,\dots ,X_n]/(X_1^p,\dots ,X_n^p)$ denote the truncated polynomial ring in $n$ variables over $k$. We write $x_i$ for the image of $X_i$ in $A(n)$. The $n$th Witt-Jacobson algebra $W(n)$ is defined as the Lie algebra of derivations of $A(n)$. It is restricted and simple, with the $p$-map being given by ordinary multiplication in $\text{End}(A(n))$: $\partial^{[p]}=\partial^p$ for all $\partial\in W(n)$. Furthermore, it is an $A(n)$-module in an obvious way, and has a standard basis $\{x_1^{\alpha_1}\dots x_n^{\alpha_n}\partial_i\ \vert \ 0\leq \alpha_j< p,1\leq i\leq n\}$ where $\partial_i$ denotes partial differentiation with respect to $x_i$. We will often use standard multi-index notation: For an $n$-tuple $\alpha=(\alpha_1,\dots,\alpha_n)$ with $0\leq \alpha_j< p$ for $1\leq j\leq n$, we write $x^\alpha$ for $x_1^{\alpha_1}\dots x_n^{\alpha_n}$ and define the degree of $x^\alpha$ to be $\vert \alpha \vert =\alpha_1+\dots + \alpha_n$. We denote by $\epsilon_j$ the $n$-tuple $(0,\dots,0,1,0,\dots,0)$ with 1 in the $j$'th place, and by $\tau$ the $n$-tuple $(p-1,\dots,p-1)$. The commutator in $W(n)$ is given by:
\begin{equation}\label{basic}
[x^\alpha\partial_i,x^\beta\partial_j]=\beta_ix^{\alpha+\beta-\epsilon_i}\partial_j-\alpha_jx^{\alpha+\beta-\epsilon_j}\partial_i
\end{equation}
An important tool in the study of $W(n)$ is the standard grading $W(n)=\bigoplus_{i=-1}^N W(n)_i$, where $N=n(p-1)-1$ and
\begin{equation}
W(n)_i=\sum_{j=1}^n\sum_{\vert \alpha\vert =i+1} kx^\alpha\partial_j
\end{equation}    
Every exterior power $\Omega^r_{A(n)/k}$ ($r\geq 0$) of the module of differentials of $A(n)$ over $k$ is a $W(n)$-module in a natural way, which makes it possible to define the following (restricted) subalgebras:
\begin{equation}\label{S}
S(n)=\{ \partial\in W(n) \ \vert \ \partial(dx_1\wedge dx_2\wedge\dots \wedge dx_n)=0\}
\end{equation} 
\begin{equation}\label{H}
H(2m)=\{ \partial\in W(2m) \ \vert \ \partial(\sum_{i=1}^m dx_i\wedge dx_{2m+1-i})=0\}
\end{equation}
\begin{equation}\label{K}
K(2m+1)=\{ \partial\in W(2m+1)\ \vert \ \partial(\omega_K)\in k\omega_K\}
\end{equation}
Here $\omega_K=\sum_{i=1}^m(x_idx_{2m+1-i}-x_{2m+1-i}dx_i)+dx_{2m+1}$. It turns out, that in all three cases a certain higher derived algebra (to be elaborated on) is simple and restricted, and the three families of simple restricted Lie algebras obtained in this way are known respectively as the special algebras, the Hamiltonian algebras and the contact algebras. Together with the Witt-Jacobson algebras they constitute the \textit{simple restricted Lie algebras of Cartan type}. \\
Let us now gather some facts about $S(n)$. Define linear maps $\text{div}:W(n)\to A(n)$ and $D_{ij}:A(n)\to W(n)$ ($1\leq i,j\leq n$) by
\begin{equation}
\text{div}(\partial)=\sum_{i=1}^n\partial_i(\partial(x_i))
\end{equation}
\begin{equation}
D_{ij}(f)=\partial_j(f)\partial_i-\partial_i(f)\partial_j
\end{equation} 
for all $\partial\in W(n)$ and $f\in A(n)$. A direct calculation shows that
\begin{equation}
\partial(dx_1\wedge\dots\wedge dx_n)=\text{div}(\partial)dx_1\wedge\dots\wedge dx_n
\end{equation}
for all $\partial\in W(n)$, which implies $S(n)=\{ \partial\in W(n)\ \vert \ \text{div}(\partial)=0\}$. Using this alternative definition it is easy to see that the images of the $D_{ij}$ are contained in $S(n)$. In fact, it can be shown that:
\begin{equation}
S(n)=\sum_{i,j} D_{ij}(A(n))\oplus \bigoplus_{i=1}^n kx^{\tau-(p-1)\epsilon_i}\partial_i
\end{equation}
Furthermore, it turns out that $\sum_{i,j} D_{ij}(A(n))$ is equal to the derived algebra $S(n)^{(1)}$ of $S(n)$. If $n\geq 3$, then $S(n)^{(1)}$ is restricted and simple, but for $n=2$ the second derived algebra $S(n)^{(2)}$ is a proper ideal of $S(n)^{(1)}$, since the element $D_{12}(x_1^{p-1}x_2^{p-1})$ is not contained in the former. However, $S(n)^{(2)}=D_{12}(\sum_{\vert \alpha\vert <2p-2} kx^{\alpha})$ is restricted and simple, and we define the $n$th \textit{special algebra} to be $S(n)^{(1+\delta_{n2})}$. It is not hard to see that $S(1)$ is one-dimensional and therefore not terribly interesting, so in the following we will always assume $n\geq 2$ when looking at the special algebras.

Let us move on to the family $H(2m)$. Note first that $H(2)=S(2)$, so it makes no harm to assume $m\geq 2$. For a number $i\in\{1,\dots,2m\}$ we set $i'=2m+1-i$ and
\begin{equation}   
\sigma(i)=\begin{cases} 1 & \text{if $1\leq i\leq m$} \\
-1 & \text{if $m+1\leq i\leq 2m$}
\end{cases}
\end{equation}
For $\partial=\sum_{i=1}^{2m}f_i\partial_i\in W(2m)$ the condition in (\ref{H}) can be shown to be equivalent to
\begin{equation}\label{Hcon}
\sigma(i)\partial_j(f_i)=\sigma(j')\partial_{i'}(f_{j'})
\end{equation}
for all $1\leq i,j\leq 2m$. Using this equation it is not hard to see that the image of the linear map $D_H:A(2m)\to W(2m)$ defined by 
\begin{equation}
D_H(f)=\sum_{i=1}^{2m}\sigma(i)\partial_i(f)\partial_{i'}
\end{equation}
is contained in $H(2m)$. It turns out that $H(2m)^{(1)}=D_H(\sum_{\vert\alpha\vert <2m(p-1)}kx^\alpha)$. This subalgebra is simple and restricted, and we call it a \textit{Hamiltonian algebra}. 

The special and Hamiltonian algebras are easily seen to be \text{graded} subalgebras of the corresponding Witt-Jacobson algebra, but this is not true for the last family, the contact algebras: Define a linear map $D_K:A(2m+1)\to W(2m+1)$ by $D_K(f)=\sum_{i=1}^{2m+1}f_i\partial_i$, where:
\begin{equation}
f_i=x_i\partial_{2m+1}(f)+\sigma(i')\partial_{i'}(f)\ \ \ \ \ \text{for}\ 1\leq i\leq 2m
\end{equation}
\begin{equation}
f_{2m+1}=2f-\sum_{j=1}^{2m}x_j\partial_j(f)
\end{equation}
Furthermore, we define $\Delta(f)=2f-\sum_{j=1}^{2m}x_j\partial_j(f)$ and 
\begin{equation}
\langle f,g\rangle=\Delta(f)\partial_{2m+1}(g)-\Delta(g)\partial_{2m+1}(f)+\sum_{j=1}^{2m}\sigma(j)\partial_j(f)\partial_{j'}(g)
\end{equation}
A basic calculation proves the commutation formula $[D_K(f),D_K(g)]=D_K(\langle f,g\rangle)$. It turns out that the image of $D_K$ is exactly $K(2m+1)$. Grading $A(2m+1)$ by $\text{deg}(x^\alpha)=\vert\vert \alpha\vert\vert=\vert\alpha\vert+\alpha_{2m+1}-2$ induces a grading on $K(2m+1)$ via
\begin{equation}
K(2m+1)_j=\text{span}\{D_K(x^\alpha) \ \vert \ \text{deg}(x^\alpha)=j\}
\end{equation} 
The derived algebra $K(2m+1)^{(1)}$ is restricted and simple, and we call it a \textit{contact algebra}. It can be shown that $K(2m+1)^{(1)}=\text{span}\{ D_K(x^\alpha)\ \vert \ \alpha\neq \tau\}$ if $2m+4\equiv 0 \ \text{mod}\ p$ and $K(2m+1)^{(1)}=K(2m+1)$ otherwise. We will need a few formulas regarding the product $\langle \cdot,\cdot\rangle$ (compare \cite{BGP}, p. 57, but note that formula (v) there is not correct as stated):
\begin{equation}\label{1}
\langle 1,x^\alpha\rangle=\alpha_{2m+1}x^{\alpha-\epsilon_{2m+1}}
\end{equation}
\begin{equation}\label{2}
\langle x_i,x^\alpha\rangle =\sigma(i)\alpha_{i'}x^{\alpha-\epsilon_{i'}}+\alpha_{2m+1}x^{\alpha+\epsilon_i-\epsilon_{2m+1}}\ \ \ \text{for}\ 1\leq i\leq 2m
\end{equation}
\begin{equation}\label{3}
\langle x_{2m+1},x^\alpha\rangle =\vert\vert \alpha\vert\vert x^\alpha 
\end{equation}
\begin{equation}\label{4}
\langle x_ix_j,x^\alpha\rangle=\sigma(i)\alpha_{i'}x^{\alpha+\epsilon_j-\epsilon_{i'}}+\sigma(j)\alpha_{j'}x^{\alpha+\epsilon_i-\epsilon_{j'}}\ \ \ \text{for}\ 1\leq i,j\leq 2m 
\end{equation}
\begin{equation}\label{5}
\langle x_ix_{i'},x^\alpha\rangle=(\alpha_{i'}-\alpha_i)x^\alpha\ \ \ \text{for}\ 1\leq i\leq m
\end{equation}
From now on we will (by abuse of notation) write $W(n)$, $S(n)$, $H(n)$ and $K(n)$ for the corresponding simple derived subalgebra, with the convention that $n=2m$ for the Hamiltonian type and $n=2m+1$ for the contact type. We let $L$ denote an arbitrary restricted algebra of Cartan type and use the notation $\widehat{L}$ for the algebra from which it is derived (with the convention $\widehat{W(n)}=W(n))$. Then we have the following lemma, which sums up some of the most important information from the preceding discussion:
\begin{lemma}\label{start}
Let $A(n)$ be graded in the usual way if $L\in\{W,S,H\}$ and by $\textit{deg}(x^\beta)=\vert\vert\beta\vert\vert$ if $L\in\{K\}$. Then there exists a finite family of graded linear maps $D_{\alpha}:A(n)\to \widehat{L}$ such that $L$ is spanned by elements of the form $D_\alpha(x^\beta)$. The maps $\{D_\alpha\}$ are said to be associated to $L$. Furthermore, if $L\in\{W,S,H\}$ we have
\begin{equation}\label{dif}
(\mathrm{ad}\ \partial_s)\circ D_\alpha=D_\alpha\circ \partial_s
\end{equation}
for $1\leq s\leq n$ and all $\alpha$.
\end{lemma}
\begin{proof}
For $W(n)$ we can use the maps $D_i$ defined by $D_i(f)=f\partial_i$ for all $f\in A(n)$. For $S(n)$ we use the $D_{ij}$ with $i\neq j$, and for $H(n)$, $K(n)$ we use $D_H$, $D_K$ respectively. The identity (\ref{dif}) is an easy consequence of the formula $[\partial_s,x^\beta\partial_j]=\partial_s(x^\beta)\partial_j$, which follows from (\ref{basic}). For example, if $L$ is of type $H$, we get:
\begin{equation}
[\partial_s,D_H(f)]=\sum_{i=1}^{2m}\sigma(i)\partial_s(\partial_i(f))\partial_{i'}=\sum_{i=1}^{2m}\sigma(i)\partial_i(\partial_s(f))\partial_{i'}=D_H(\partial_s(f)).
\end{equation} 
\end{proof}
One final preliminary remark: The sum of components of negative degree in $L$ turns up in several of our proofs, so it is nice to have a concrete description (which can be derived from the information above): If $L\in\{ W,S,H\}$ then 
\begin{equation}
\bigoplus_{i<0} L_i=L_{-1}=\text{span}\{\partial_1,\dots,\partial_n\},
\end{equation}
and if $L\in\{K\}$ then
\begin{equation}
\bigoplus_{i<0}L_i=L_{-2}\oplus L_{-1}=\text{span}\{ D_K(1),D_K(x_1),\dots, D_K(x_{2m})\}.
\end{equation}

\section{The Coadjoint Representation}
Let $G$ denote the automorphism group of $L$, then we have a canonical action on $L^*$ given by
\begin{equation}
g.\chi(x)=\chi(g^{-1}(x))
\end{equation}
for all $g\in G$, $\chi\in L^*$ and $x\in L$. This is the \textit{coadjoint representation} of $G$. As mentioned in the introduction we aim to show that the invariant ring $k[L^*]^G$ is trivial, i.e., $k[L^*]^G=k$. For this, we need a few facts about $G$ (see \cite{W}). The following subgroups will be very important:
\begin{equation}
G_0=\{g\in G\ \vert \ g(L_i)=L_i\ \text{for all}\ i\}
\end{equation}
\begin{equation}
G_r=\{g\in G\ \vert \ g(x)-x\in L_{\geq r+i}\ \text{for all $i$ and all $x\in L_i$}\}
\end{equation}
We have $G=G_0\ltimes G_1$ with $G_0\cong GL_n$ if $L\in \{W,S\}$ and $G_0\cong CSp_{2m}$ if $L=H(2m)$ or $L=K(2m+1)$. If we define $L_{\geq i}=\bigoplus_{j\geq i}L_j$, then it is a consequence of the semidirect product decomposition $G=G_0\ltimes G_1$ that $g(L_{\geq i})=L_{\geq i}$ for all $g\in G$ and all $i$. The grading on $L$ induces a grading $L^*=\bigoplus_i L^*_i$ by setting $L^*_i=\{\chi\in L^*\ \vert\ \chi(L_j)=0\ \text{for all}\ j\neq i\}$. For any $\chi\in L^*$ we write $\chi_i$ for the component of $\chi$ of degree $i$ and $\chi_-$ for the sum of components of negative degree. Set $L^*_{\leq i}=\bigoplus_{j\leq i}L^*_j$, then it follows from $g(L_{\geq i})=L_{\geq i}$ and the definition of the coadjoint action, that $g(L^*_{\leq i})=L^*_{\leq i}$ for all $g\in G$ and all $i$.

Inside $G_0$ we have a copy of $k^*$ corresponding to the scalar matrices, and this subgroup turns out to be of crucial importance. An easy calculation shows that the action of $k^*$ is given by $t.\chi=t^{-i}\chi$ for all $t\in k^*$, $\chi\in L_i^*$. Another important ingredient in the proof of triviality of $k[L^*]^G$ is the following set:
\begin{equation}
Y=\{ \chi\in L^*\ \vert \ \text{there exists}\ g\in G\ \text{such that}\ (g.\chi)_-=0\}
\end{equation}
For the proof of the next lemma we will need an alternate grading on $L$, defined as follows: Grade $A(n)$ and $W(n)$ by $\text{deg}(x^\alpha)=\sum_{i=1}^n i\alpha_i$ and $\text{deg}(x^\alpha\partial_j)=\sum_{i=1}^n i\alpha_i-j$. Formula (\ref{basic}) shows that this grades $W(n)$ as a Lie algebra, and we write $W_{[s]}$ for the $s$th graded component. Looking at the definitions, we see that the associated maps $\{D_\alpha\}$ are all graded ($D_i$ of degree $-i$, $D_{ij}$ of degree $-i-j$ and $D_H,D_K$ of degree $-n$), which implies that we get a Lie algebra grading on $L$ by setting $L_{[s]}=L\cap W_{[s]}$. Furthermore, each $L_i$ is graded, i.e., $L_i=\bigoplus_s(L_i\cap L_{[s]})$.
  
\begin{lemma}\label{first} 
For every $\chi\in L_{\leq 1}^*\backslash L_{\leq 0}^*$ there exists $g\in G_2$ such that $g.\chi=\chi_0+\chi_1$.
\end{lemma}
\begin{proof} Let $\{D_\alpha\}$ be the maps associated to $L$. The core of the proof is an adaptation (and simplification) of Theorem 4.1(3) in \cite{HZ}. Assume $\chi\in L_{\leq 1}^*\backslash L_{\leq 0}^*$ and note that it is enough to find $g\in G_2$ such that $(g.\chi)_-=0$. For if $y\in L_{\geq 0}$ then $g^{-1}(y)-y\in L_{\geq 2}$, which implies that $g.\chi$ and $\chi$ agree on $L_{\geq 0}$, i.e., $g.\chi=\chi_0+\chi_1$. 

Choose $t$ minimal such that $\chi(L_1\cap L_{[t]})\neq 0$, then we can find an associated map $D$ and some $x^\beta\in A(n)$ such that $x=D(x^\beta)\in L_1\cap L_{[t]}$ and $\chi(x)\neq 0$. We deal first with the case $L\in\{ W,S,H\}$: If $\chi_-=0$ there is nothing to show, so assume otherwise and choose $l$ maximal with the property $\chi(\partial_l)\neq 0$. Define $E=D(x^{\beta+\epsilon_l})$, then we have $E\in L_2\cap L_{[t+l]}$, and according to Theorem 1 in \cite{W} we can find, for any $c\in k$, a $g\in G_2$ such that
\begin{equation}\label{WSH}
g^{-1}(\partial_s)-\partial_s-[cE,\partial_s]\in L_{\geq 2}
\end{equation}
for all $s$, $1\leq s\leq n$. This implies $g.\chi(\partial_s)=\chi(\partial_s)+c\chi([E,\partial_s])$. Since $[E,\partial_s]\in L_1\cap L_{[t+l-s]}$ we get $g.\chi(\partial_s)=\chi(\partial_s)=0$ if $s>l$, by minimality of $t$ and maximality of $l$. Notice also that:
\begin{equation}\label{27}
[E,\partial_l]=[D(x^{\beta+\epsilon_l}),\partial_l]=-D(\partial_l(x^{\beta+\epsilon_l}))=-(\beta_l+1)x
\end{equation}
Here the second equality follows from (\ref{dif}). Note that $\beta_l+1\neq 0$ because $\beta_l\leq 3$ (here the assumption $p>3$ comes into play), so the calculation implies $\chi([E,\partial_l])\neq 0$, which again implies that we can choose $c$ such that $g.\chi(\partial_l)=0$. Applying this process at most $n$ times and composing the $g$'s we find, we end up with a $g'\in G_2$ such that $(g'.\chi)_-=0$.

Now assume $L$ is a contact algebra, in which case $\bigoplus_{i<0}L_i=L_{-2}\oplus L_{-1}=\text{span}\{ D_K(1),D_K(x_1),\dots, D_K(x_{2m})\}$. If $\chi(L_{-1})\neq 0$ we choose $l$ \textit{minimal} such that $\chi(D_K(x_l))\neq 0$ and define $E=D_K(x^{\beta+\epsilon_{l'}})\in L_2\cap L_{[t+l']}$. Again we can find $g\in G_2$ such that (\ref{WSH}) holds with $D_K(x_s)$ in place of $\partial_s$, and it follows that $g.\chi(D_K(x_s))=\chi(D_K(x_s))+c\chi([E,D_K(x_s)])$. Since $[E,D_K(x_s)]\in L_1\cap L_{[t+l'-s']}$ we get $\chi(D_K(x_s))=0$ for $s<l$ because of the choice of $t$. Furthermore:
\begin{equation}\label{21}
[E,D_K(x_l)]=D_K(\langle x^{\beta+\epsilon_{l'}},x_l\rangle)=-\sigma(l)(\beta_{l'}+1)D_K(x^\beta)-\beta_nD_K(x^{\beta+\epsilon_l+\epsilon_{l'}-\epsilon_n})
\end{equation}
If $\chi(D_K(x^{\beta+\epsilon_l+\epsilon_{l'}-\epsilon_n}))=0$ the proof proceeds as in the first case. Otherwise we can replace $x^\beta$ by $x^{\beta+\epsilon_l+\epsilon_{l'}-\epsilon_n}\in L_1\cap L_{[t]}$ and repeat the process. The new $x^\beta$ satisfies $\beta_n=0$, so the last term in (\ref{21}) disappears and we can again proceed as in the first case. Now induction yields a $g'\in G_2$ such that $(g'.\chi)_{-1}=0$. Finally, let $E'=D_K(x^{\beta+\epsilon_n})\in L_3$. Then we can find $g''\in G_3$ such that:
\begin{align*}
&(g''g').\chi(D_K(1))=\\&g'.\chi(D_K(1))+cg'.\chi([E',D_K(1)])=g'.\chi(D_K(1))-(\beta_n+1)cg'.\chi(D_K(x^\beta))
\end{align*}
It is clear that we can again choose $c$ such that $(g''g').\chi(D_K(1))=0$, and since $((g''g').\chi)_{-1}=(g'.\chi)_{-1}=0$ (because $g''\in G_3$) we are done. 
\end{proof}

\begin{corollary}\label{kor}
For every $\chi\in L_{\leq 1}^*\backslash L_{\leq 0}^*$ we have $\{\chi_0+t\chi_1\ \vert \ t\in k^*\}\subseteq G.\chi$ and $\chi_0\in \overline{G.\chi}$.
\end{corollary}
\begin{proof}
Use the action of $k^*$ on $\chi_0+\chi_1$ and take the limit as $t$ approaches zero.
\end{proof}

\begin{lemma}\label{vigtig}
The set $Y$ is dense in $L^*$.
\end{lemma}
\begin{proof}
Consider an element $y=\sum_{i=1}^n a_iD(x^{\tau-\epsilon_i})$ with $D=D_1$ if $L=W(n)$, $D=D_{12}$ if $L=S(n)$, $D=D_H$ if $L=H(n)$ and $D=D_K$ if $L=K(n)$ (the $a_i\in k$ are arbitrary). Using the grading on $L$ and the assumption $p\geq 5$ one checks that, with the exception of the case $L=W(1)$, $p=5$, we have $(\text{ad}\ y)^2=0$ and $[(\text{ad}\ y)(x_1),(\text{ad}\ y) (x_2)]=0$ for all $x_1,x_2\in L$, which implies that $g=\text{exp}(\text{ad}\ y)=\text{id}+\text{ad}\ y$ is an automorphism of $L$ (if $L=W(1)$ one can use the results on orbit representatives in \cite{M} to prove $Y=L^*$). We treat the case $L\in \{W,S,H\}$ first. Here we can use (\ref{dif}) to get:
\begin{equation}
g(\partial_s)=\partial_s+[y,\partial_s]=\partial_s-\sum_{i=1}^n a_iD(\partial_s(x^{\tau-\epsilon_i}))
\end{equation}
For any $\chi\in L^*$ we set $\chi(D(\partial_s(x^{\tau-\epsilon_i})))=b_{si}$. So we have:
\begin{equation}
g^{-1}.\chi(\partial_s)=\chi(\partial_s)-\sum_{i=1}^n a_ib_{si}
\end{equation}
If the $n\times n$ matrix $B=(b_{si})$ is invertible, then we can choose the $a_i$ such that $g^{-1}.\chi(\partial_s)=0$ for all $s$. The set $Y'\subseteq Y$ of all $\chi\in L^*$ that satisfies this condition is open, so we just have to show that it is nonempty: Define matrices $C=(\partial_s(x^{\tau-\epsilon_i}))_{s,i}$ and $B'=(D(\partial_s(x^{\tau-\epsilon_i})))_{s,i}$. Explicitly, we have:
\begin{equation*}
c_{si}=\begin{cases} (p-2)x^{\tau-2\epsilon_i} & \text{if $s=i$}\\
(p-1)x^{\tau-\epsilon_i-\epsilon_s} & \text{if $s\neq i$}
\end{cases}
\end{equation*} 
We see that $C$, and therefore also $B'$, is symmetric. Furthermore, the $c_{si}$ with $s\geq i$ are linearly independent, and since $D$ is injective on the ideal generated by $x_1x_2\dots x_n$ (can easily be checked case by case) which contains all the $c_{si}$, the elements on and below the diagonal of $B'$ must also be linearly independent. But then we can choose $\chi$ such that (say) $B=I$, and we are done. 

Now let $L$ be of type $K$ and write $z_s=D_K(x_{s'})$ for $1\leq s\leq 2m$, $z_n=D_K(1)$. As before, we calculate (with the convention $x_{n'}=1$):
\begin{align*}
&g(z_s)=z_s+[y,z_s]=\\&z_s-\sum_{i=1}^na_i[D_K(x_{s'}),D_K(x^{\tau-\epsilon_i})]=z_s-\sum_{i=1}^n a_i D_K(\langle x_{s'},x^{\tau-\epsilon_i}\rangle)
\end{align*}
We set $b_{si}=\chi(D_K(\langle x_{s'},x^{\tau-\epsilon_i} \rangle))$ for all $\chi\in L^*$, then it is again enough to find $\chi$ such that $B=(b_{si})$ is invertible: Consider first the matrix $C=(\langle x_{s'},x^{\tau-\epsilon_i} \rangle)_{s,i}$. Using (\ref{1}), (\ref{2}) and (\ref{3}) we get, for $1\leq s,i\leq 2m$,
\begin{equation*}
c_{si}=\begin{cases} \sigma(i')(p-2)x^{\tau-2\epsilon_i} & \text{if $s=i$} \\
\sigma(s')(p-1)x^{\tau-\epsilon_i-\epsilon_{s}}+\delta_{si'}(p-1)x^{\tau-\epsilon_n} & \text{if $s\neq i$}
\end{cases}
\end{equation*}
and:
\begin{equation*}
c_{sn}=\sigma(s')(p-1)x^{\tau-\epsilon_s-\epsilon_n}
\end{equation*}
\begin{equation*}
c_{ns}=(p-1)x^{\tau-\epsilon_s-\epsilon_n}
\end{equation*}
\begin{equation*}
c_{nn}=(p-2)x^{\tau-2\epsilon_n}
\end{equation*}
It is easy to see from these formulas that the $c_{si}$ with $s\geq i$ are linearly independent. But $D_K$ is injective (can be derived from \cite{SF}, Lemma 5.1), so the entries on and below the diagonal in the matrix $B'=(D_K(\langle x_{s'},x^{\tau-\epsilon_i}\rangle))_{s,i}$ are also linearly independent. Then we can choose $\chi$ such that $B$ has the form:

\[ \left(
\begin{matrix}
1 &\ & \Huge{*}\\
\ & \ddots & \ \\
\Huge{0} & \ & 1 
\end{matrix}
\right) \] \\
This matrix is clearly invertible.
\end{proof} 

\begin{lemma}\label{sidst} 
There exists $x\in L_1$ such that the map $(\mathrm{ad}\ x)_{\vert L_0}:L_0\to L_1$ is injective.
\end{lemma}
\begin{proof}
Assume first that $L\in\{W,S\}$. We set $x=\sum_{i=1}^{n-1} x_i^2\partial_i-\sum_{i=1}^{n-1}2x_nx_i\partial_n\in L_1$. If $y=\sum_{i,j} b_{ij}x_i\partial_j$ is an arbitrary element of $L_0$ and $[x,y]=0$, then we have for $1\leq s\leq n-1$:
\begin{equation}
y\circ x(x_s)=\sum_{i=1}^n2b_{is}x_ix_s
\end{equation}
\begin{equation}
x\circ y(x_s)=\sum_{i=1}^{n-1}b_{is}x_i^2-\sum_{i=1}^{n-1}2b_{ns}x_ix_n
\end{equation}
It follows (since we assume $[x,y]=0$) that $b_{ij}=0$ if $j\neq n$. So $y=\sum_{i=1}^nb_{in}x_i\partial_n$ and:
\begin{equation}
y\circ x(x_n)=\sum_{i=1}^{n-1}\sum_{j=1}^n 2b_{jn}x_i x_j
\end{equation}
\begin{equation}
x\circ y(x_n)=\sum_{j=1}^{n-1} b_{jn}x_j^2-\sum_{i=1}^{n-1}2b_{nn}x_ix_n
\end{equation}
Equating these expressions, we get $b_{jn}=0$ for $1\leq j\leq n$, so $y=0$ and $x$ works like it should.

If $L$ is of type $H$ we set $x=\sum_{i=1}^{2n} \sigma(i)x_i^2\partial_{i'}\in L_1$. Note first that the condition (\ref{Hcon}) with $i=j$ implies that if $y=\sum_{i,j}b_{ij}x_i\partial_j\in L$, then $b_{ii}=-b_{i'i'}$. Again we calculate, for $1\leq s\leq 2n$:
\begin{equation}
y\circ x(x_s)=\sum_{i=1}^{2n} 2\sigma(s')b_{is'}x_ix_{s'}
\end{equation}
\begin{equation}
x\circ y(x_s)=\sum_{i=1}^{2n}\sigma(i')b_{is}x_{i'}^2
\end{equation}
If $[x,y]=0$ then these equations show that all $b_{ij}=0$ (using that $b_{ss}=-b_{s's'}$), so $(\mathrm{ad}\ x)_{\vert L_0}$ is injective, and we are done.

Finally, assume $L$ is of type $K$ and set $x=D_K(\sum_{s=1}^{2m}x_s^3)\in L_1$. An arbitrary element $y$ of $L_0$ has the form $D_K(\sum_{1\leq i\leq j\leq 2m}b_{ij}x_ix_j+cx_n)$. Using formulas (\ref{3}), (\ref{4}) and (\ref{5}) we get:
\begin{align*}
& \langle \sum_{1\leq i\leq j\leq 2m}b_{ij}x_ix_j+cx_n,\sum_{s=1}^{2m}x_s^3\rangle =\\&\sum_{\substack{1\leq i\leq j\leq 2m \\ j\neq i'}}3b_{ij}(\sigma(i)x_jx_{i'}^2+\sigma(j)x_ix_{j'}^2)+\sum_{i=1}^m 3b_{ii'}(x_{i'}^3-x_i^3)+c\sum_{i=1}^{2m} x_i^3 
\end{align*}
Each term in the first sum contains standard basis elements which do not appear anywhere else, so if $[x,y]=0$ all the $b_{ij}$ in the first sum must also be zero (recall that $D_K$ is injective). Then we can look at the last two sums to get $3b_{ii'}=-c=-3b_{ii'}$ for $1\leq i\leq m$, which implies $b_{ii'}=0$. Finally $c$ must also be zero, and we are done. 
\end{proof}

\begin{theorem}\label{finale}
$k[L^*]^G=k$.
\end{theorem}
\begin{proof} Let $f\in k[L^*]^G$. We show first that $f$ is constant on $L^*_0$: For any $\chi\in L_1^*$ and $g\in G$, we use Corollary \ref{kor} to write:
\begin{equation}\label{snart}
f(0)=f(\chi_0)=f(\chi)=f(g.\chi)=f((g.\chi)_0)
\end{equation}
Let $\chi'\in L_0^*$ be arbitrary. For a basis $\{y_s\}$ of $L_0$ we can find $g\in G_1$ that satisfies
\begin{equation}
g^{-1}(y_s)-y_s-[x,y_s]\in L_{\geq 2}
\end{equation}  
for all $s$, where $x$ is the one from Lemma \ref{sidst}. But the $[x,y_s]$ are linearly independent, so we can choose $\chi\in L_1^*$ that satisfies $\chi([x,y_s])=\chi'(y_s)$ for all $s$. This means that
\begin{equation}
g.\chi(y_s)=\chi(y_s)+\chi([x,y_s])=\chi'(y_s)
\end{equation}
and $(g.\chi)_0=\chi'$. Now (\ref{snart}) gives $f(\chi')=f(0)$.

For any $\chi\in Y$, we can find $g\in G$ such that $(g.\chi)_-=0$ and use the action of $k^*$ to get:
\begin{equation}
\{(g.\chi)_0+t(g.\chi)_1+\dots +t^N(g.\chi)_N\ \vert \ t\in k^*\}\subseteq G.\chi
\end{equation} 
It follows, by taking the limit as $t$ approaches zero, that $(g.\chi)_0\in\overline{G.\chi}$, and so $f(\chi)=f((g.\chi)_0)=f(0)$. We have shown that $f$ is constant on $Y$, and the theorem follows from Lemma \ref{vigtig}. 
\end{proof}
 
It is a consequence of the classification theorem for finite-dimensional simple modular Lie algebras (completed in \cite{PS}) and Chevalley's restriction theorem that the property in Theorem $\ref{finale}$ actually characterises the restricted Cartan type algebras among the simple restricted Lie algebras:
\begin{corollary}
Let $L$ be a restricted simple Lie algebra over a field of characteristic $p>5$ and let $G$ be the identity component of $\mathrm{Aut}(L)$. Then $L$ is of Cartan type if and only if $k[L^*]^G=k$.
\end{corollary} 

\begin{proof}[Acknowledgements] I would like to thank my advisor Professor Jens Carsten Jantzen for providing key ideas and insight.
 \end{proof}

\end{document}